\documentclass[twocolumn]{autart}  
\usepackage{graphicx}
\usepackage{epsfig}
\usepackage{epstopdf}
\usepackage{amsmath,amssymb}
\usepackage{xcolor,color}
\usepackage{float}
\newtheorem{remark}{Remark}
\newcommand{\X}{{\bar{x}}}
\newcommand{\A}{{\bar{A}}}
\newcommand{\B}{{\bar{B}}}
\newcommand{\Hb}{{\bar{H}}}
\newcommand{\E}{{\mathbb{E}}}
\newcommand{\AC}{{\mathcal{A}}}
\newcommand{\I}{{\mathcal{I}}}

\newcommand{\Sb}{{\mathbb{S}}}
\newcommand{\Pb}{{\bar{P}}}
\newcommand{\CL}{{\mathcal{J}}}
\newcommand{\C}{{\bar{C}}}
\newcommand{\OM}{{\bar{O}}}

\DeclareMathOperator{\rank}{rank}
\newenvironment{proof}
{\noindent\textbf{Proof.}}
{\hfill$\blacksquare$\par}

\begin{document}

\begin{frontmatter}

\title{Risk Sensitive Filtering for Singular Systems subject to Round-Robin Protocol \thanksref{footnoteinfo}} 

\thanks[footnoteinfo]{This paper was not presented at any IFAC 
meeting. Corresponding author N.~K.~Tomar. \ead{nktomar@iitp.ac.in}}

\author[add1]{Ashna Goel}\ead{ashna\_2021ma10@iitp.ac.in}, 
\author[add2]{Shovan Bhaumik}\ead{shovan.bhaumik@iitp.ac.in},     
\author[add1]{Nutan Kumar Tomar}\ead{nktomar@iitp.ac.in} 

\address[add1]{Department of Mathematics, Indian Institute of Technology Patna, India} \address[add2]{Department of Electrical Engineering, Indian Institute of Technology Patna, India}      

\begin{keyword}       
Singular systems; Networked control systems; Round-robin protocol; Weierstrass canonical form; Kalman filtering; Risk sensitive filtering
\end{keyword}    

\begin{abstract} 
This paper develops a risk sensitive (RS) Kalman filtering framework for discrete-time linear stochastic singular systems operating under communication constraints imposed by a round-robin protocol. Due to limited network bandwidth, only a subset of the available measurements can be transmitted at each sampling instant, resulting in a periodically varying measurement structure. By employing the Weierstrass canonical form (WCF), the singular system is transformed into an equivalent augmented state space model, yielding a round-robin induced periodic system (RRIPS). A recursive risk sensitive Kalman filter (RSKF) is then developed for the RRIPS through a Bayesian formulation and the minimization of an exponential quadratic cost function, from which the recursive filtering equations are obtained for the original singular system. To enhance robustness against modeling uncertainties and disturbances, an adaptive RS mechanism is introduced in which the risk parameter is adjusted online according to the available covariance information. This adaptive strategy guarantees the positive definiteness of the predicted covariance matrix while adjusting the degree of risk sensitivity to the prevailing estimation uncertainty. Furthermore, sufficient conditions ensuring the filter stability are established using the observability and controllability concepts of periodic systems. The proposed framework reduces to the standard KF for singular systems when the RS parameter vanishes and recovers the standard RSKF when the singular matrix reduces to the identity matrix. Finally, numerical results are presented to demonstrate the effectiveness, robustness, and improved estimation performance of the proposed approach in comparison with the standard KF.
\end{abstract}

\end{frontmatter}

\section{Introduction} \label{sec:intro}
State estimation is a fundamental task in modern control and signal processing systems, where the complete state vector is often inaccessible through direct measurement. For linear stochastic systems represented in the state space framework, the KF has emerged as one of the most widely adopted estimation techniques due to its recursive structure and optimal minimum mean square error performance under Gaussian noise assumptions \cite{kalman1960new}, \cite{anderson2012optimal}. Consequently, KF based methods have found widespread applications in aerospace systems \cite{schmidt1981kalman}, robotics \cite{hartley2020contact}, industrial automation \cite{garriz2019development}, and power networks \cite{hoffmann2013minimal}. With the increasing integration of communication technologies into control architectures, many practical systems now rely on shared communication networks to exchange information among sensors, controllers, and actuators, leading to the emergence of Networked Control Systems (NCS) \cite{gupta2009networked}. Compared with conventional control architectures, NCSs offer several advantages, such as lower installation and maintenance costs, improved flexibility, ease of system expansion, and support for geographically distributed sensing and control \cite{zhang2001stability}, \cite{zhang2019networked}.

Despite these benefits, communication networks introduces several challenges that can degrade estimation and control performance. Typical network-induced effects include transmission delays \cite{zhang2012network}, packet dropouts \cite{wu2007design}, quantization effects \cite{xia2011analysis}, and bandwidth limitations \cite{wang2023review}. In particular, when multiple sensors nodes share a common communication channel, it is generally impossible to transmit all measurement data simultaneously. To regulate network access and efficiently utilize communication resources, various scheduling protocols have been proposed, including the Round-Robin Protocol (RRP) \cite{fan2025distributed}, the Try-Once-Discard (TOD) protocol \cite{wang2021observer}, and other deterministic or stochastic scheduling schemes \cite{zou2016observer}. Among these, the RRP is particularly attractive due to its simplicity, fairness, and deterministic periodic structure, in which sensor measurements are transmitted sequentially according to a predefined schedule. 

Motivated by the challenges arising in NCSs, considerable research efforts have been devoted to state estimation under round-robin scheduling. Existing studies have investigated the dynamics by round-robin induced systems by modeling them as delayed switched system and establishing stability conditions using Lyapunov-based methods \cite{liu2016quantized}. The observability and exponential stability of the RRIPS have also been examined in \cite{zhang2006communication}. Furthermore, state estimation problem for NCSs subject to packet disorders under the round-robin scheduling, as well as sensor scheduling schemes have been addressed in \cite{liu2021recursive}, \cite{shi2011sensor}. Building upon these developments, several estimation algorithms have been developed for remotely available measurements \cite{liu2015robust}, \cite{gao2007cal}. 

Despite their effectiveness, KF based estimators rely heavily on accurate system models and noise statistics. In the presence of modeling uncertainties, parameter variations, and unmodeled dynamics their estimation performance may deteriorate significantly. This limitation has motivated the development of robust estimation methods, among which the $H_\infty$ filter and the RSKF are particularly prominent. While the $H_\infty$ filter attenuates the worst-case effect of disturbances, the RSKF minimizes an exponential quadratic cost function that imposes a larger penalty on large estimation errors \cite{sadhu2009particle}, \cite{xie1994robust}. The foundations of RS estimation were established by the pioneering works of \cite{ho1964bayesian}, while its robustness properties were rigorously established in \cite{boel2002robustness}. Owing to its enhanced robustness against model uncertainty, RS filtering has attracted considerable attention in stochastic and networked systems, and has been extended to systems with delayed measurements and communication constraints \cite{tiwari2022risk}. Nevertheless, compared with the extensive literature on Kalman and $H_\infty$ filtering for NCSs, RS filtering under explicit communication scheduling protocols, particularly the RRP, remains relatively unexplored.

While the aforementioned studies primarily focused on state space systems, many practical processes are more naturally described by singular or descriptor systems. Such systems arise in the presence of algebraic constraints coupled with dynamic equations and are frequently encountered in electrical circuits, power systems, constrained mechanical systems, and chemical processes \cite{belov2018control}, \cite{goel2025state}. Due to their ability to capture both dynamic and algebraic relationships, singular systems provide a more general modeling framework, with a special case of state space systems. However, the singular structure introduces additional challenges in state estimation, since the standard Kalman filtering recursion is not directly applicable. Following the pioneering work of \cite{dai1987state}, extensive research has been devoted to state estimation for singular systems including Kalman filtering \cite{dai1989filtering}, unbiased minimum variance estimation \cite{yu2019optimal}, observer design \cite{hou2002observer}, and robust $H_\infty$ filtering \cite{yue2004robust}. These studies have established a solid theoretical foundation for robust filtering of singular systems under uncertainty.

The state estimation problem becomes even more challenging when singular systems operate over communication network. Consequently, increasing attention has been directed toward state estimation for networked singular systems. Existing works have addressed estimation problems in the presence of packet dropouts, communication delays, missing measurements, and limited bandwidth \cite{feng2011descriptor}. Distributed filtering, fusion estimation, and $H_\infty$ filtering methods for networked singular systems have also been reported in \cite{li2013h}, \cite{zheng2023distributed}. More recently, a Kalman-type recursive estimator was developed for singular systems operating under a round-robin communication protocol, where measurements are transmitted sequentially to accommodate bandwidth limitations \cite{goel2024kalman}. 

Despite the significant advances in state estimation for NCSs and networked singular systems, an important research gap still exists. Although RS filtering has been extensively studied for state space systems and shown to provide enhanced robustness against modeling uncertainties and disturbances, its extension to networked singular systems remains largely unexplored \cite{zhang2003risk}. Moreover, existing estimation results for singular systems operating under round-robin scheduling are largely confined to the standard Kalman filtering framework and minimum variance estimation \cite{wan2018h_}, \cite{li2024secure}. To the best of the authors' knowledge, RS filtering for singular systems subject to RRP has not yet been investigated. This gap is particularly important because networked singular systems are simultaneously affected by communication constraints and the inherent complexities of singular systems.

Motivated by this gap, this paper develops a RS Kalman filtering framework for discrete-time stochastic singular systems operating under RRP. By employing the WCF, the original singular system is transformed into an equivalent RRIPS, for which a recursive RSKF is derived using a Bayesian formulation and an exponential quadratic cost criterion. Furthermore, an adaptive RS mechanism is introduced to improve robustness against modeling uncertainties, and sufficient conditions are established for filter boundedness and stability. Numerical examples are provided to demonstrate the effectiveness of the proposed approach in comparison with the standard KF.

The remainder of the paper is organized as follows. Section \ref{sec:probform} presents the necessary preliminaries on singular systems and formulates the state estimation problem under RRP. Section \ref{sec:rskf} develops the proposed RS Kalman filtering framework for singular systems subject to RRP. The observability and stability properties of the proposed filter are analyzed in Section \ref{sec:stability}. Numerical examples illustrating the effectiveness of the proposed filter are provided in Section \ref{sec:simu}. Finally, concluding remarks and future research directions are presented in Section \ref{sec:concl}.

\section{Problem Formulation} \label{sec:probform}
\subsection{System transformation} \label{subsec:systrans}
Consider the class of discrete-time stochastic singular systems described by
\begin{subequations} \label{sls}
	\begin{align}
E x_{k+1} &= (A+\Delta A) x_k + B w_k, \label{slsa} \\
	y_k &= Hx_k + v_k,\label{slsb}
	\end{align}
\end{subequations}
where  $E \in \mathbb{R}^{n \times n}$ is a singular matrix with $\rank(E)<n$, and $A \in \mathbb{R}^{n \times n}$, $B \in \mathbb{R}^{n \times r}$ and $H \in \mathbb{R}^{q \times n}$ are known constant matrices. The vectors $x_k \in \mathbb{R}^n$ and $y_k \in \mathbb{R}^q$ denote the system state and measurement output at time instant $k$, respectively. The vectors $w_k \in \mathbb{R}^r$ and $v_k \in \mathbb{R}^q$ represent the process and measurement noises, respectively. It is assumed that the uncertainty term $\Delta A$ is deterministic but unknown to the estimator. As a result, the filter has access only to the nominal dynamics and uses the model of the form
\begin{equation} \label{nominalsls}
E x_{k+1}=Ax_k+Bw_k.
\end{equation}
In contrast, the measurement equation is presumed to be known without uncertainty and is therefore same as \eqref{slsb}. Owing to the singularity of the matrix $E$, the vector $x_k$ cannot be regarded as a conventional state vector for arbitrary initial conditions, \emph{cf}. Eq.~\eqref{fsol}. However, $x_k$ completely characterizes the system behavior at time instant $k$. An initial condition is said to be consistent if there exists a solution trajectory satisfying \eqref{slsa} and \eqref{nominalsls}. The following assumptions are adopted throughout the paper.

\noindent \textbf{Assumption $1$}: The matrix pair $(E, A)$ is regular, \emph{i.e.}, $\det(\lambda E-A)\not\equiv 0$, where $\det(\lambda E-A)$ is not identically zero as a polynomial in $\lambda$.

\noindent \textbf{Assumption $2$}: The process noise $w_k$ and measurement noise $v_k$ are assumed to be zero mean Gaussian random vectors satisfying $\E[w_k{w_k}^\top]=Q_k, \; \E[v_k{v_k}^\top]=R_k, \; \text{and}  ~\E[w_k{v_j}^\top]=0 \enspace \forall k$ and $j$.

\noindent \textbf{Assumption $3$}: The initial state $x_0$ is assumed to be a zero-mean random vector with known covariance. Moreover, $x_0$ is independent of $v_k$, but may be correlated with $w_0$ and a finite number of future noise terms $w_{k+l}$ due to the non-causal nature of the system. 

Under Assumption $1$, the regularity of the matrix pair $(E,A)$ ensures the existence of a unique solution to \eqref{slsa} for every consistent initial condition \cite{duan2010analysis}. It is well known that the matrix pair $(E, A)$ is regular if and only if there exist non-singular matrices $U, V \in \mathbb{C}^{n \times n}$ such that 
	\begin{equation}\label{WCF}
	UEV =\begin{bmatrix}
		I_{n_1} & 0 \\
		0 & N
	\end{bmatrix} \text{\;and\;} UAV=\begin{bmatrix}
		A_1 & 0 \\
		0 & I_{n_2}
	\end{bmatrix},
	\end{equation}
where $n_1$ is the degree of the polynomial $\det(\lambda E-A)$, $n_2 = n - n_1$, $A_1 \in \mathbb{R}^{n_1 \times n_1}$ and $N \in \mathbb{R}^{n_2 \times n_2}$ is a nilpotent matrix with nilpotency index $h$, namely, $N^h = 0$. The above form \eqref{WCF} is referred to as the WCF of a matrix pair $(E, A)$.  The reader is referred to \cite{berger2012quasi} for more details on the WCF and the construction of the transformation matrices $U$ and $V$.

Let $UB=\begin{bmatrix}
	B_{1}^\top  &
	B_{2}^\top
\end{bmatrix}^\top$ and $ HV=\begin{bmatrix}
	H_{1} & H_{2}
\end{bmatrix}$, where the partitions conform to the dimensions $n_1$ and $n_2$. By introducing the coordinate transformation $x_{k} = V \begin{bmatrix}
	x_{1,k} \\
	x_{2,k}
\end{bmatrix}$, the system \eqref{sls} can be decomposed into
\begin{subequations}\label{wcfsys}
	\begin{align}
		x_{1,k+1} &= A_{1} x_{1,k}+B_{1} w_k,\label{wcfsysa}\\
		N x_{2,k+1} &= x_{2,k}+B_{2} w_k, \label{wcfsysb}\\
		y_k &= H_{1} x_{1,k}+H_{2} x_{2,k}+v_k, \label{wcfsysc}
	\end{align}
\end{subequations} 
where $x_{1,k} \in \mathbb{R}^{n_1}$ and $x_{2,k} \in \mathbb{R}^{n_2}$. The subsystem \eqref{wcfsysa} represents the dynamic part of the system in state space form, whereas subsystem \eqref{wcfsysb} represents its algebraic component. Due to the nilpotency of $N$, the subsystem \eqref{wcfsysb} admits an explicit solution \cite[pp.~47]{belov2018control}. As a result, the solutions of \eqref{wcfsysa} and \eqref{wcfsysb} can be expressed as
\begin{subequations}
	\begin{align}
		x_{1,k} &= A_{1}^{k} x_{1,0} + \sum_{i = 0}^{k-1}A_{1}^{k-i-1}B_1 w_i,\label{ssol}\\
		x_{2,k} &= -\sum_{i = 0}^{h-1}N^{i}B_2 w_{k+i}. \label{fsol}
	\end{align}
\end{subequations} 

By combining the expressions for $x_{1,k}$ and $x_{2,k}$, the system \eqref{wcfsys} can be transformed into the following system \cite{zhang2003risk} 
\begin{subequations} \label{asys}
\begin{align}
	\X_{k+1} &= \A \X_{k} + \B(z) w_{k}, \label{asysa}\\
	y_k &= \Hb \X_{k} + v_{k}, \label{asysb}
\end{align}
\end{subequations}
where $\X_k = \begin{bmatrix}
    x_{1,k} \\
    x_{2,k}
\end{bmatrix}$ and the matrices are
\begin{eqnarray*}
\A &=& \begin{bmatrix}
	A_1 & 0\\
	0 & 0
\end{bmatrix}, \; \Hb = \begin{bmatrix}
	H_{1} & H_2
\end{bmatrix}, \\
\B(z) &=& B^{(0)} + B^{(1)}z + \dots + B^{(h)}z^h , \\
B^{(0)} &=& \begin{bmatrix}
    B_1 \\
    0
\end{bmatrix}, B^{(i)} = -\begin{bmatrix}
    0 \\
    N^{i-1}B_2
\end{bmatrix}, i = 1,2,\dots,h
\end{eqnarray*}
and $z$ is the forward shift operator, \emph{i.e.} $z w_k = w_{k+1}$.

\begin{remark} \label{rem1}
It should be noted that the process noise covariance associated with the model \eqref{asys} is given by $\bar{Q}_k = \E[(\B(z)w_k)(\B(z)w_k)^\top]$. By expanding $\B(z)w_k$, the equivalent covariance can be expressed as $\bar{Q}_k = \sum \limits_{i=0}^{h}\sum \limits_{j=0}^{h}
B^{(i)}
\E[w_{k+i}w_{k+j}^\top]
(B^{(j)})^\top$ which is equivalent to $\bar{Q}_k =
\sum \limits_{i=0}^{h} B^{(i)}Q_{k+i}(B^{(i)})^\top$. Therefore, despite the presence of forward shift operator and non-causal structure of \eqref{asys}, the process noise covariance admits an equivalent matrix representation.
\end{remark}

\begin{remark}\label{rem2}
It follows from \eqref{asys}, that the resulting model is non-singular but non-causal. For the causal subsystem \eqref{wcfsysa}, the initial state $x_{1,0}$ is assumed to be a zero-mean random vector with covariance $P_1$, and is uncorrelated with the noise sequences $w_k$ and $v_k$. In contrast, \eqref{wcfsysb} describes the non-causal subsystem. From \eqref{fsol}, the corresponding initial state can be expressed as $x_{2,0} = -\sum \limits_{i = 0}^{h-1}N^{i}B_2 w_{i}$. Hence, $x_{2,0}$ is a zero mean random vector with covariance $\E[x_{2,0}x_{2,0}^{\top}]
= P_2 = \sum \limits_{i=1}^{h-1}
N^{i-1}B_2 Q_i (N^{i-1}B_2)^\top.$ Since $x_0
= V\bar{x}_0 = V\begin{bmatrix}
x_{1,0} \\
x_{2,0}
\end{bmatrix},$
the initial state $x_0$ is also zero mean with covariance $\E[x_0x_0^\top] =
V \operatorname{diag}\{P_1,P_2\} V^{\top}.$
Furthermore, due to the non-causal structure in \eqref{fsol}, the initial state $x_0$ is correlated with process noises $w_0, w_1, \dots, w_{h-1}$. This property distinguishes singular systems from state space systems, where the initial state is typically assumed to be independent of both process and measurement noises.
\end{remark}

\subsection{Measurement transmitted with round-robin protocol} \label{subsec:rrp}
In this paper, the estimation center is assumed to be far from the sensors, and the measurement data are transmitted through a shared communication network. Due to the limited bandwidth of the communication channel,  it is assumed that the complete measurement vector $y_k$ cannot be transmitted at any instant $k$. To address this limitation, the measurement vector is partitioned, and each part of the measurements is transmitted according to a predefined schedule. Specifically, a round-robin protocol is adopted to overcome the communication capacity constraints. 

Suppose that, at each time instant $k$, the communication channel can transmit only $q'$ measurement components out of the total $q$ components. Define the period of RRP as $\sigma =\lceil\frac{q}{q'}\rceil$, where $\lceil . \rceil$ denotes a ceiling function that returns the smallest successive integer. Then, the transmission schedule is characterized by the periodic index as  $m_k = \bmod (k, \sigma)+1$, where $\bmod (k, \sigma)$ represents the non-negative remainder obtained upon dividing $k$ by $\sigma$. Since $m_k\in\phi_{\sigma}$ with $\phi_l=\{1,2,\ldots,l\}$, the index $m_k$ cycles through the set $\phi_{\sigma}$ in a periodic manner. Thus, based on the capacity of the network channel, the measurement vector $y_k$ is partitioned into subsets, and only the subset associated with the period $\sigma$ and index $m_k$ is transmitted.  Finally, we create partitions $y_{k}^{m_k}$ of the measurement vector $y_k$ as follows:
\begin{equation}\label{partition}
y_k^{m_k}=
\begin{cases}
\begin{bmatrix}
y_{k,a_1} & y_{k,a_2} & \cdots & y_{k,a_{q'}}
\end{bmatrix}^{\top}, & m_k\in\phi_{\sigma-1},\\
\begin{bmatrix}
y_{k,a_1} & y_{k,a_2} & \cdots & y_{k,q}
\end{bmatrix}^{\top}, & m_k=\sigma.
\end{cases}
\end{equation}
where $j = 1,2,\dots,q'$, $a_j=(m_k-1)q'+j$, and $y_{k,a_j}$ is the $a_j^{th}$ component of the vector $y_k$. It should be noted that the partition corresponding to $m_k = \sigma$ may contain fewer than $q'$ measurement components. In this situation, additional entries selected from the set $\{y_{k,1},y_{k,2},\ldots,y_{k,(\sigma-1)q'}\}$ are appended so that the dimension of $y_k^{\sigma}$ becomes $q'$ \cite{xu2022extended}.

Now, based on the communication capacity of the network channel, the period $\sigma$ and the scheduling index $m_k$ are specified. Subsequently, the measurement vector $y_k$ in \eqref{asys} is partitioned into subsets $y_k^{m_k}$ as defined in \eqref{partition}. Thus the system \eqref{asys} can be represented by the following round-robin induced periodic system (RRIPS):
\begin{subequations} \label{RRIPS}
\begin{align}
\X_{k+1} &= \A \X_{k} + \B(z) w_{k}, \label{RRIPSa}\\
y_k^{m_k} &= \Hb_{m_k}  \X_{k} + v_k^{m_k}, \label{RRIPSb}
\end{align}
\end{subequations} 
where $v_k^{m_k}$ is formed by partitioning $v_k$ in the same manner that as $y_k^{m_k}$ is generated from $y_k$, and $\Hb_{m_k}$ consists of the corresponding rows of $\Hb$, \emph{i.e.}, 
$$
\begin{aligned}
\Hb_{m_k} =
\begin{cases}
[\Hb_{m_1}^\top \; \Hb_{m_2}^\top \; \dots \; \Hb_{m_{q'}}^\top]^\top,
& \text{for } m_k \in \phi_{\sigma-1}, \\[2mm]
[\Hb_{m_1}^\top \; \dots \; \Hb_{q}^\top]^\top,
& \text{for } m_k = \sigma.
\end{cases}
\end{aligned}
$$
where $\Hb_{m_j}$ is the $m_j^{th}$ row of matrix $\Hb$.

\section{Risk Sensitive State Estimation} \label{sec:rskf}
The objective of this paper is to construct a recursive optimal estimate $\hat{x}_{k|k}$ using the available measurements $y_{1:k}$. To achieve this goal, the estimation problem is formulated for the RRIPS \eqref{RRIPS}, from which the estimate of the system \eqref{sls} can subsequently be obtained. Within the RS framework, the estimator is obtained by minimizing an exponential cost that penalizes estimation errors according to prescribed risk parameters. Accordingly, the optimal estimate at each time step is obtained by minimizing the cost function  
\begin{equation} 
\begin{aligned}
\CL_k(\hat{\X}_{k | k} | y_{1: k}^{m_k})&=\E [\exp (\sum_{i=1}^{k-1} \mu_{1,i}\rho_1(e_{i|i}) +\mu_{2,k}\rho_2(e_{k|k})], \label{jkcost}
   \end{aligned}
\end{equation}
where $\hat{\X}_{i|i}$ denotes the posterior estimate and $e_{i|i} = \X_i-\hat{\X}_{i|i}$ denotes the posterior estimation error at time instant $i$. The parameters $\mu_1\ge 0$ and $\mu_2>0$ characterize the degree of risk sensitivity associated with past and current estimation errors, respectively. Moreover, the functions $\rho_1(\cdot)$ and $\rho_2(\cdot)$ are both strictly convex, continuous and bounded from below, attaining global minima at $0$. In particular, the minimum RS estimate (MRSE) is defined by
\begin{equation} \label{argmin}
    \hat{\X}_{k|k}^\star=\arg \min _{\hat{\X}_{k|k}} \CL_k(\hat{\X}_{k|k}|y_{1:k}^{m_k}).
\end{equation}
In this paper, to obtain a recursive filtering algorithm, the convex functions $\rho_1(\cdot)$ and $\rho_2(\cdot)$ are chosen as quadratic forms, namely, $\rho_1(e_{i|i})=e_{i|i}^\top e_{i|i}$ and $\rho_2(e_{k|k})=e_{k|k}^\top e_{k|k}$. Under this choice, the cost function \eqref{jkcost} becomes
\begin{equation}
\begin{aligned}
\CL_k(\hat{\X}_{k | k} | y_{1: k}^{m_k})&=\E [\exp (\sum_{i=1}^{k-1} \mu_{1,i}e_{i|i}^\top e_{i|i} +\mu_{2,k} e_{k|k}^\top e_{k|k})], \label{jkcostquad}
   \end{aligned}
\end{equation}

\subsection{Bayesian formulation of the MRSE}\label{subsec:genframe}
For the development of the proposed estimator, the optimization problem \eqref{argmin} is formulated within a Bayesian estimation framework. Let $y_{1:k}^{m_k} = \{y_1^{m_k}, \dots, y_k^{m_k}\}$ denote the available measurement sequence at the estimation center up to time instant $k$. Then, the posterior probability density of $\X_{0:k}$ conditioned on $y_{1: k}^{m_k}$, can be expressed as \cite{arulampalam2002tutorial} 
\begin{equation*}
    p(\X_{0: k}|y_{1: k}^{m_k})=\dfrac{ p(y_k^{m_k}|\X_{0: k}, y_{1: k-1}^{m_k}) p(\X_{0: k}|y_{1: k-1}^{m_k})}{p(y_k^{m_k} | y_{1: k-1}^{m_k})},
\end{equation*}
with 
\begin{equation*}
  p(y_k^{m_k} | y_{1: k-1}^{m_k})=\int p(y_k^{m_k} |\X_{0: k}, y_{1: k-1}^{m_k}) p(\X_{0: k} | y_{1: k-1}^{m_k}) d \X_{0: k},
\end{equation*}
According to the measurement model \eqref{asysb}, the current measurement $y_k^{m_k}$ is completely determined by the current state $\X_k$. Therefore, conditioned on $\X_k$, the measurement $y_k^{m_k}$ is independent of the past states $\X_{0: k-1}$ as well as the previous received measurements $y_{1:k-1}^{m_k}$. This conditional independence property implies that 
\begin{equation*}
p(y_k^{m_k}|\X_{0: k}, y_{1:k-1}^{m_k})=p(y_k^{m_k} | \X_k).
\end{equation*}
Assuming that the optimal estimates $\hat{\X}_{0|0}^\star, \ldots, \hat{\X}_{k-1|k-1}^\star$ are available up to time instant $k-1$, the posterior density of state $\X_k$ can be written as 
\begin{equation*}
 p(\X_k|y_{1:k}^{m_k})=\dfrac{p(y_k^{m_k}|\X_k) p(\X_k|y_{1: k-1}^{m_k})}{p(y_k^{m_k}|y_{1: k-1}^{m_k})}.   
\end{equation*}
Using the Chapman-Kolmogorov relation for $p(\X_k|y_{1:k-1}^{m_k})$, the posterior density can be written as
\begin{equation} \label{xpostpdf}
\begin{aligned}
p(\X_k|y_{1:k}^{m_k})= \dfrac{p(y_k^{m_k}|\X_k) \int p(\X_k|\X_{k-1})p(\X_{k-1}|y_{1:k-1}^{m_k}) d \X_{k-1}}{p(y_k^{m_k}|y_{1: k-1}^{m_k})}.
\end{aligned}
\end{equation}
To account for the accumulated risk associated with past estimation errors, introduce the information state as 
\begin{equation} \label{infstate}
    \Phi_k \triangleq p(\X_k | \I_k)=\exp (\sum_{i=0}^{k-1} \mu_{1,i}e_{i|i}^{\top} e_{i | i}) p(\X_k | y_{1: k}^{m_k}),
\end{equation}
where $\I_k=\{y_{1: k}^{m_k}, e_{1|1}, ...,e_{k-1|k-1}\}$ denotes the information set at time instant $k$ \cite{bar2001estimation} and the initialization is taken as $\Phi_0=p(\X_0)$. Now, substituting \eqref{xpostpdf} into \eqref{infstate} yields
\begin{equation*}
\begin{aligned}
\Phi_k
=&\dfrac{p(y_k^{m_k}|\X_k)}
       {p(y_k^{m_k}|y_{1:k-1}^{m_k})}
\int p(\X_k|\X_{k-1}) \exp(\sum_{i=0}^{k-2}\mu_{1,i}
e_{i|i}^{\top}e_{i|i}) \\
& \exp(
\mu_{1,k-1} e_{k-1|k-1}^{\top}e_{k-1|k-1}) 
p(\X_{k-1}|y_{1:k-1}^{m_k})\,d\X_{k-1},
\end{aligned}
\end{equation*}
which is equivalent to
\begin{equation} \label{rho}
\begin{aligned}
\Phi_k
=&\dfrac{p(y_k^{m_k}|\X_k)}
       {p(y_k^{m_k}|y_{1:k-1}^{m_k})}
\int p(\X_k|\X_{k-1}) \\
& \exp(\mu_{1,k-1} e_{k-1|k-1}^{\top}e_{k-1|k-1}) \Phi_{k-1}\, d\X_{k-1}.
\end{aligned}
\end{equation}
Define the predicted information state, $p(\X_k | \bar{\I}_{k-1})$, as
\begin{equation} \label{xpredicted}
\begin{aligned}
p(\X_k | \bar{\I}_{k-1})= & \int p(\X_k | \X_{k-1}) \exp(\mu_{1, k-1} e_{k-1 | k-1}^{\top} e_{k-1 | k-1}) \\
& \times \Phi_{k-1} \, d \X_{k-1},
\end{aligned}
\end{equation}
where $\bar{\I}_{k-1}=\{\I_{k-1}, e_{k-1 | k-1}\}$ denotes the prior information set.

Using \eqref{xpredicted}, the information state $\Phi_k$ can be written as
\begin{equation} \label{reducedrho}
\Phi_k \triangleq p(\X_k | \I_k)= \dfrac{p(y_k^{m_k}|\X_k)p(\X_k | \bar{\I}_{k-1})}{p(y_k^{m_k}|y_{1:k-1}^{m_k})} .
\end{equation}
Using the definition of expectation in the cost function \eqref{jkcostquad} yields
\begin{equation*}
\begin{aligned}
    \CL_k(\hat{\X}_{k|k}|y_{1:k}^{m_k}) =& \int \exp (\sum_{i=1}^{k-1} \mu_{1, i} e_{i | i}^{\top} e_{i | i}+\mu_{2, k} e_{k | k}^{\top} e_{k|k}) \\
    & \times p(\X_k | y_{1: k}^{m_k})\, d \X_k,
\end{aligned}    
\end{equation*}
and by \eqref{infstate}, it reduces to
\begin{equation} \label{reducedjk}
    \CL_k(\hat{\X}_{k|k}|y_{1: k}^{m_k})=\int \exp (\mu_{2, k} e_{k|k}^{\top} e_{k|k}) \Phi_k \, d \X_k.
\end{equation}
Therefore, the minimum RS estimate is obtained by recursively solving \eqref{argmin} together with \eqref{xpredicted}-\eqref{reducedjk}. It is important to note that the past estimation errors are embedded in the information state $\Phi_k$, whereas the cost function in \eqref{reducedjk} depends solely on the current estimation error.

\begin{remark} \label{rem3}
For $\mu_{1,k-1}=0$ and $\mu_{2,k}>0$, no risk penalty is assigned to the previous estimation errors. In this case, the information state coincides with the posterior density, namely, $p(\X_k|\I_k)=p(\X_k|y_{1: k}^{m_k})$, and the cost function $\CL_k(\hat{\X}_{k|k}|y_{1: k}^{m_k})$ reduces to a standard exponential quadratic form.
\end{remark} 

\subsection{Risk sensitive filtering for $\sigma$-periodic system} \label{subsec:rskfaugsys}
The objective of this subsection is to obtain a recursive expression for the estimate $\hat{\X}_{k|k}$ of the RRIPS \eqref{RRIPS} through the minimization of $\CL_k(\hat{\X}_{k|k}|y_{1:k}^{m_k})$. To obtain a recursive solution, $\Phi_k$ is assumed to be an unnormalized Gaussian density for sufficiently small nonnegative scalar $\mu_{1, k-1}$ \cite{boel2002robustness}. Since $\Phi_0=p(\X_0)$ is Gaussian, this property is preserved recursively and $\Phi_k$ remains Gaussian whenever $\Phi_{k-1}$ is Gaussian. 

\begin{thm} \label{theorem1}
For the RRIPS \eqref{RRIPS}, the predicted mean and error covariance can be obtained as
\begin{equation} \label{prexhat}
\begin{aligned}
\hat{\X}_{k|k-1} &= \A \hat{\X}_{k-1|k-1}, \\
\Pb_{k|k-1} &= \A(\Pb_{k-1|k-1}^{-1}-2 \mu_{1, k-1} I)^{-1} \A^{\top}+\bar{Q}_k ,
\end{aligned}
\end{equation}
where $\Pb_{k-1|k-1}$ is the posterior error covariance.
\end{thm}

\begin{proof}
To determine the prediction expressions, consider the predicted information state given in \eqref{xpredicted}. Now, suppose that $\Phi_{k-1}$ admits the Gaussian distribution, then we can write 
\begin{equation*}
    \begin{aligned}
\Phi_{k-1} = \dfrac{\exp(-\frac{1}{2}e_{k-1|k-1}^{\top}\Pb_{k-1|k-1}^{-1}e_{k-1|k-1})}{p(y_{k-1}^{m_k}|y_{1:k-2}^{m_k})\sqrt{(2 \pi)^n|\Pb_{k-1|k-1}|}} 
\end{aligned}
\end{equation*}
where $e_{k-1|k-1} = \X_{k-1}-\hat{\X}_{k-1|k-1}$ is the estimation error. Substituting the above expression into \eqref{xpredicted} and combining the quadratic terms appearing in the exponent yields
\begin{equation} \label{xpredgaussain1}
\begin{aligned}
\Phi_{k-1}
=&\dfrac{((2\pi)^n|\Pb_{k-1|k-1}|)^{-1/2}}
{p(y_{k-1}^{m_k}|y_{1:k-2}^{m_k})}\int p(\X_k|\X_{k-1})
\\
& \times \exp(
-\frac{1}{2}e_{k-1|k-1}^{\top}\Omega_{k-1}
e_{k-1|k-1})
\,d\X_{k-1}.
\end{aligned}
\end{equation}
where $\Omega_{k-1} = \Pb_{k-1|k-1}^{-1}-2\mu_{1,k-1}I$ and the risk parameter $\mu_{1, k-1}$ is selected such that $2 \mu_{1, k-1} \Pb_{k-1|k-1}<I$. Under this condition, the exponential term in \eqref{xpredgaussain1} can be interpreted as a Gaussian density up to a normalization constant. Hence,
\begin{equation} \label{xpredgaussain2}
\begin{aligned}
p(\X_k|\bar{\I}_{k-1}) &= \dfrac{\int p(\X_k|\X_{k-1}) \mathcal{N}(\X_{k-1} ; \hat{\X}_{k-1 | k-1},\Omega_{k-1}^{-1}) d \X_{k-1}}{p(y_{k-1}^{m_k}|y_{1:k-2}^{m_k})}.
\end{aligned}
\end{equation}
Using the state transition model associated with \eqref{asys}, we have $p(\X_k|\X_{k-1})=\mathcal{N}(\X_k ; \A \X_{k-1}, \bar{Q}_k)$ \cite{ho1964bayesian} and substituting it into \eqref{xpredgaussain2}, the Gaussian product identity \cite{challa2011fundamentals} yields 
\begin{equation} \label{xpredgaussain3}
\begin{aligned}
p(\X_k|\tilde{\I}_{k-1}) &=\dfrac{\int \mathcal{N}(\X_k ; \Pi_1, \Gamma_1)\mathcal{N}(\X_{k-1} ; \Pi_2, \Gamma_2) d \X_{k-1}}{p(y_{k-1}^{m_k}|y_{1:k-2}^{m_k})},
\end{aligned}
\end{equation}
where $\Pi_1 = \A \hat{\X}_{k-1|k-1}$, $\Gamma_1 = \A\Omega_{k-1}^{-1} \A^{\top}+ \bar{Q}_k$, $\Lambda =\Omega_{k-1}^{-1} \A^{\top} \Sb_1^{-1}$, $\Pi_2 =\hat{\X}_{k-1|k-1}+ \Lambda(\X_k-\A \hat{\X}_{k-1|k-1})$, $\Gamma_2 = (I-\Lambda\A)\Omega_{k-1}^{-1}$.

Since the Gaussian density involving $\X_k$ is independent of $\X_{k-1}$, it can be taken outside the integral in \eqref{xpredgaussain3}. Using $\int \mathcal{N}(\X_{k-1} ; \cdot)\,d \X_{k-1}=1$, the predicted information state in \eqref{xpredgaussain3} can be obtained, up to normalization, as
\begin{equation} \label{predxgauss}
p(\X_k|\tilde{\I}_{k-1}) \sim \mathcal{N}(\X_k ; \A \hat{\X}_{k-1|k-1}, \A \Omega_{k-1}^{-1} \A^{\top}+\bar{Q}_k).
\end{equation}
The prediction formulas in \eqref{prexhat} follows directly from the corresponding mean and covariance of \eqref{predxgauss}.
\end{proof}

To derive the posterior information state $p(\X_k|\I_k)$, it is necessary to obtain expressions for the conditional mean of the measurement, the innovation covariance, and the cross covariance. The corresponding results are established in the following lemmas.

\begin{lem} \label{lemma2}
For the measurement $y_k^{m_k}$, the conditional mean and covariance conditioned on $\bar{\I}_{k-1}$ are obtained as
\begin{equation} \label{inncov}
\begin{aligned}
\E[y_k^{m_k}|\bar{\I}_{k-1}] &= \Hb_{m_k}\hat{\X}_{k|k-1}, \\
\Pb_{yy}^{m_k} &= \Hb_{m_k} \Pb_{k|k-1} \Hb_{m_k}^\top + R_k^{m_k}    
\end{aligned}
\end{equation}
where $R_k^{m_k} = \E[v_k^{m_k}{v_k^{m_k}}^\top]$ is the partitioned measurement noise covariance.
\end{lem}

\begin{proof}  
Since the measurement $y_k^{m_k}$ is independent of the past estimation errors $e_{1|1}, \dots, e_{k-1|k-1}$, the conditional expectation with respect to $\bar{\I}_{k-1}$ is equivalent to conditioning on $y_{1:k-1}^{m_k}$. Consequently, using the measurement model \eqref{asysb}, we obtain
\begin{equation*}
\begin{aligned}
\E[y_k^{m_k}|\tilde{\I}_{k-1}]
&= \E[\Hb_{m_k} \X_k + v_k^{m_k}|y_{1:k-1}^{m_k}] = \Hb_{m_k} \hat{\X}_{k|k-1}.
\end{aligned}
\end{equation*}
Next, the innovation covariance matrix is defined as 
\begin{equation*}
\Pb_{yy}^{m_k} = \E[(\bar{y}_k^{m_k})(\bar{y}_k^{m_k})^\top],
\end{equation*}
where $\bar{y}_k^{m_k} = y_k^{m_k} - \E[y_k^{m_k}|\bar{\I}_{k-1}]$ is the innovation error. By substituting $y_k^{m_k}$ and $\E[y_k^{m_k}|\bar{\I}_{k-1}]$ derived above, it follows that
\begin{equation*}
\begin{aligned}
\Pb_{yy}^{m_k} &= \E[(\Hb_{m_k} \X_k + v_k^{m_k} - \Hb_{m_k} \hat{\X}_{k|k-1}) \\
& \quad \times (\Hb_{m_k} \X_k + v_k^{m_k} - \Hb_{m_k} \hat{\X}_{k|k-1})^\top]\\
&= \E[(\Hb_{m_k} e_{k|k-1} + v_k^{m_k})(\Hb_{m_k} e_{k|k-1} + v_k^{m_k})^\top]
\end{aligned}
\end{equation*}
where $e_{k|k-1} = \X_k - \hat{\X}_{k|k-1}$ represents the predicted error. Using the fact that $e_{k|k-1}$ is uncorrelated with $v_k^{m_k}$, together with the prescribed noise statistics, we obtain
\begin{equation} \label{sigmayy}
\Pb_{yy}^{m_k} = \Hb_{m_k} \Pb_{k|k-1} \Hb_{m_k}^\top + R_k^{m_k}.
\end{equation}
\end{proof}

\begin{lem} \label{lemma3}
The cross covariance matrix is obtained as $\Pb_{\X y}^{m_k} = \Pb_{k|k-1}\Hb_{m_k}^\top.$
\end{lem}

\begin{proof}  
From the definition of $\Pb_{\X y}^{m_k} = \E[(\X_k - \hat{\X}_{k|k-1})(\bar{y}_k^{m_k})^\top]$ and using the derivation steps analogous to those in Lemma \ref{lemma2}, we obtain
\begin{equation}
\begin{aligned}
\Pb_{\X y}^{m_k} &= \E[e_{k|k-1}
(\Hb_{m_k} \X_k + v_k^{m_k} - \Hb_{m_k} \hat{\X}_{k|k-1})^\top] \\
 \Pb_{\X y}^{m_k} &= \E[e_{k|k-1}
(\Hb_{m_k} e_{k|k-1} + v_k^{m_k})^\top] \\
&=\Pb_{k|k-1}\Hb_{m_k}^\top.
\end{aligned}
\end{equation}
\end{proof}

\begin{thm} \label{theorem4}
For the RRIPS \eqref{RRIPS}, the posterior estimate and error covariance are obtained as
\begin{equation} \label{postxhat}
\begin{aligned}
\hat{\X}_{k|k} &= \hat{\X}_{k|k-1} + \Pb_{\X y}^{m_k}(\Pb_{yy}^{m_k})^{-1} (y_k^{m_k} - \Hb_{m_k} \hat{\X}_{k|k-1}), \\
\Pb_{k|k} &= \Pb_{k|k-1} - \Pb_{\X y}^{m_k} (\Pb_{yy}^{m_k})^{-1} (\Pb_{\X y}^{m_k})^\top.
\end{aligned}
\end{equation}
\end{thm}

\begin{proof}  
The posterior information density can be obtained by incorporating the newly received measurement $y_k^{m_k}$ into the predicted information state $p(\X_k|\bar{\I}_{k-1})$. In view of \eqref{reducedrho}, we have
\begin{equation*} 
 p(\X_k | \I_k)= \dfrac{p(\X_k, y_k^{m_k}| \bar{\I}_{k-1})}{p(y_k^{m_k}|y_{1:k-1}^{m_k})}.
\end{equation*}
Assuming that $p(y_k^{m_k} | \bar{\I}_{k-1})$ is Gaussian, \emph{i.e.,} 
$$p(y_k^{m_k} |\bar{\I}_{k-1}) \sim \mathcal{N}(y_k^{m_k};\, \E[y_k^{m_k}|\bar{\I}_{k-1}], \Pb_{yy}^{m_k}),$$ 
where $\E[y_k^{m_k}|\bar{\I}_{k-1}]$ is obtained from Lemma \ref{lemma2}. Furthermore, Eqs. \eqref{prexhat} and \eqref{predxgauss} imply that $p(\X_k |\bar{\I}_{k-1}) \sim \mathcal{N}(\X_k;\, \hat{\X}_{k|k-1}, \Pb_{k|k-1})$. Since both the predicted state and the measurement are jointly Gaussian, their joint Gaussian density can be written as
\begin{equation} \label{postxgauss}
p(\X_k|\I_k)
= \dfrac{\mathcal{N}\left(
\begin{bmatrix}
y_k^{m_k} \\
\X_k
\end{bmatrix};
\begin{bmatrix}
\E[y_k^{m_k} |\bar{\I}_{k-1}] \\
\hat{\X}_{k|k-1}
\end{bmatrix},
\begin{bmatrix}
\Pb_{yy}^{m_k} & \Pb_{y \X}^{m_k} \\
\Pb_{\X y}^{m_k} & \Pb_{\X \X}
\end{bmatrix}
\right)}{p(y_k^{m_k}|y_{1: k-1}^{m_k})},
\end{equation}
where $\Pb_{\X \X} = \Pb_{k|k-1}$. By rearranging the quadratic terms appearing in \eqref{postxgauss} and completing the square (for more details, we refer to \cite{tiwari2022risk}), the posterior density, $p(\X_k|\I_k)$,  can be expressed in the form
\begin{equation}\label{postxexp}
    \begin{aligned}
p(\X_k |\I_k) &= \dfrac{\exp(
-\frac{1}{2}\hat{e}_{k|k}^\top\Theta^{-1}\hat{e}_{k|k})}{p(y_k^{m_k}|y_{1: k-1}^{m_k})}.  
    \end{aligned}
\end{equation}
where $\hat{e}_{k|k} = e_{k|k-1} - \Pb_{\X y}^{m_k} (\Pb_{yy}^{m_k})^{-1}\bar{y}_k^{m_k}$ and $\Theta =  \Pb_{\X \X} - \Pb_{\X y}^{m_k}(\Pb_{yy}^{m_k})^{-1} \Pb_{y\X}^{m_k}$.

The condition $2\mu_{1,k-1}\Pb_{k-1|k-1}<I$, together with the nonsingularity of the covariance matrices, guarantees that \eqref{postxexp} defines a Gaussian density. Therefore, the posterior mean and covariance are obtained from the parameters of $p(\X_k|\I_k)$. Finally, the update equations in \eqref{postxhat} are obtained by combining \eqref{postxexp} with the results established in Lemma \ref{lemma2}.
\end{proof}

\begin{remark} \label{rem4}
It is worth noting that when $\mu_{1,k-1} = 0$, the contribution of the past estimation errors vanishes from the criterion \eqref{jkcostquad}, and the resulting cost function reduces to the standard exponential quadratic form. In this case, the proposed RSKF degenerates to the standard KF. Moreover, if the update estimate is written as $\hat{\X}_{k|k} = L_k \hat{\X}_{k|k-1} + K_k y_k^{m_k}$, where $L_k = I - K_k \Hb,~K_k = \Pb_{\X y}^{m_k} (\Pb_{yy}^{m_k})^{-1}$, then, from the Proposition $1$ of \cite{ray1993state}, the proposed estimator remains unbiased, \emph{i.e.}, $\E[\X_k - \hat{\X}_{k|k}] =  0$, provided that $\mu_{2,k} > 0$ and $\mu_{1,k-1}$ is selected as a sufficiently small non-negative scalar satisfying $2\mu_{1,k-1}\Pb_{k-1|k-1} < I$, $\forall \; k$. 
\end{remark} 

\begin{remark}
Once the estimate $\hat{\X}_{k|k}$ of the transformed state vector is obtained from the proposed filter for the RRIPS \eqref{RRIPS}, the estimate of the original state $x_k$ can be obtained as 
	\begin{equation} \label{os}
	\hat{x}_{k|k} = V \hat{\X}_{k|k},
	\end{equation} 
where $V$ is the WCF transformation matrix. Therefore, the state estimation problem for the singular system \eqref{sls} can be addressed in the transformed coordinates by estimating $\X_k$, after which the corresponding estimate in the original state is recovered through the transformation matrix $V$.
\end{remark}

\subsection{Selection of risk sensitive parameter} \label{subsec:rspara}
Most existing RS filtering approaches, such as \cite{boel2002robustness}, \cite{dey2002risk}, employ fixed values of the RS parameters throughout the estimation process. However, such a restriction is not essential for the proposed framework. Specifically, the RS parameter $\mu_{2,k}$ affects only the scaling of the cost function and does not alter the resulting estimate provided that $\mu_{2,k}>0$. In contrast, the RS parameter $\mu_{1,k-1}$ determines the degree to which past estimation errors influence the current estimation update. To ensure the well-posedness of the filter recursion, the predicted covariance matrix must remain positive definite. This requires $\Pb_{k-1|k-1}^{-1} - 2\mu_{1,k-1} I > 0$, which imposes a covariance dependent constraint on the admissible values of $\mu_{1,k-1}$. Motivated by this observation, an adaptive selection mechanism is adopted in which $\mu_{1,k-1}$ is updated online using the current covariance information. As a result, the level of risk sensitivity is adjusted dynamically according to the prevailing estimation uncertainty.

The proposed adaptation offers two key benefits. On the one hand, it ensures that the predicted covariance matrix remains positive definite, thereby preserving the feasibility of the recursive filtering procedure. On the other hand, it allows the degree of risk sensitivity to vary with the current level of estimation uncertainty. As a result, the filter exhibits enhanced robustness and achieve improved estimation accuracy relative to the standard RSKF methods which employs the fixed RS parameters.

\section{Observability and Stability} \label{sec:stability}
Before discussing the stability of the proposed RSKF, the next theorem provides a characterization for observability of \eqref{RRIPS}. The proof proceeds along the same lines as that of Proposition $4.4$ in \cite{bittanti2009periodic} and hence it is omitted. 

\begin{thm} 
The $\sigma$-periodic system \eqref{RRIPS} is observable at time instant $k$, if and only if, for each eigenpair $(\varepsilon, \omega)$ of $\A^{\sigma}$, the conditions
$\A^{\sigma} \omega = \varepsilon \omega \enspace \text{and} \enspace \Hb_{m_s} \A^{s-k} \omega =0$ imply that $\omega = 0 \enspace \forall s \in \{k, k+1, \cdots k+\sigma-1\}$.
	\end{thm}

The filter developed in this paper is based on the nominal augmented model \eqref{RRIPS}. To analyze the effect of model uncertainty on the estimation performance, we now consider the uncertain singular system \eqref{sls}. Applying the concept of WCF as described in Section \ref{subsec:systrans}, the uncertain system can be transformed as the following augmented state space representation
\begin{subequations} \label{mcsys}
\begin{align}
	\X_{k+1} &= (\A + \Delta \A)\X_{k} + \B(z) w_{k}, \label{mcsysa}\\
	y_k^{m_k} &= \Hb_{m_k} \X_{k} + v_{k}^{m_k}. \label{mcsysb}
\end{align}
\end{subequations}
The matrix $\Delta\A$ denotes the uncertainty in the transformed model arising from the uncertainty $\Delta A$ in \eqref{sls}. It is assumed that $\A$, $\A+\Delta\A$, and $\Hb_{m_k}$ are bounded, and that $\A+\Delta\A$ remains invertible whenever $\A$ is invertible. Due to the nonsingularity of the WCF transformation matrices, bounded uncertainties in the original singular system induce bounded uncertainties in the augmented model \eqref{mcsys}. Consequently, the robustness analysis can be performed in the transformed state space framework without loss of generality.

Although the RSKF is an optimal filter \cite{boel2002robustness}, optimality alone is insufficient to ensure filter stability \cite{jazwinski2007stochastic}. Assume that the state trajectory of the nominal RRIPS \eqref{RRIPS} remains bounded. For the uncertain RRIPS system \eqref{mcsys}, define the transition matrices as 
\begin{equation} \label{transmat}
\begin{aligned}
\bar{\AC}_{i,k} = (\A_{i,k} + \Delta \A_{i,k}) &= ( \A + \Delta \A)^{i-k}, \\
\bar{\AC}_{k,i} =  (\A_{k,i} + \Delta \A_{k,i}) &= ( \A + \Delta \A)^{k-i}, \quad 0 \le i < k,
\end{aligned}
\end{equation}
where $\A_{k,k} = (\A_{k,k} + \Delta \A_{k,k}) = I.$ The matrices $\bar{\AC}_{i,k}$ and $\bar{\AC}_{k,i}$ describe the backward and forward propagation of the state, respectively. In addition, $\Delta \A_{i,k} = (\A + \Delta \A)^{i-k} - \A^{i-k}$, $\Delta \A_{k,i} = (\A + \Delta \A)^{k-i} - \A^{k-i}$, and whenever $\A$ is nonsingular $\A_{k-1,k} = \A_{k,k-1}^{-1} = \A^{-1}$. 

Using the above transition matrices, the observability and controllability Gramians over the interval $[k-l,k]$, where $k \geq l$ and $l \in \mathbb{Z}_>0$, are defined as
\begin{equation} 
\begin{aligned}
\OM_{k,k-l} &= \sum_{i=k-l}^{k}\bar{\AC}_{i,k}^\top \Hb_{m_k}^\top (R_i^{m_k})^{-1}\Hb_{m_k} \bar{\AC}_{i,k},\\
\C_{k,k-l} &= \sum\limits_ {i=k-l}^{k-1}\bar{\AC}_{k,i+1}^\top \bar{Q}_{i} \bar{\AC}_{k,i+1} \label{obsmatrix}.
\end{aligned}
\end{equation}
 The uncertain RRIPS \eqref{mcsys} is said to be uniformly completely observable and uniformly completely controllable, if there exists positive constants $k_1, k_2, k_3$, and $k_4$ such that 
 \begin{equation} \label{k1k2}
0 < k_1 I \le \OM_{k,k-l} \le k_2 I, \quad
0 < k_3 I \le \C_{k,k-l} \le k_4 I,
\end{equation}
for all admissible values of $k$. Equivalently, both $\OM_{k,k-l}$ and $\C_{k,k-l}$ remain uniformly bounded and positive definite over the prescribed horizon \cite{jazwinski2007stochastic}.
 
\begin{remark}
For $\Delta \A = 0$, the nominal system observability and controllability matrices reduces to
\begin{equation*}
    \begin{aligned}
O_{k,k-l} &= \sum_{i=k-l}^{k} \A_{i,k}^\top \Hb_{m_k}^\top (R_i^{m_k})^{-1}\Hb_{m_k} \A_{i,k}, \\
C_{k,k-l} &= \sum_{i=k-l}^{k-1} \A_{k,i+1}^\top \bar{Q}_{i} \A_{k,i+1}. 
    \end{aligned}
\end{equation*}
Furthermore, the boundedness and positive definiteness of $\OM_{k,k-l}$ and $\C_{k,k-l}$ imply the corresponding properties for $O_{k,k-l}$ and $C_{k,k-l}$.
\end{remark}

We now investigate the uniformly boundedness of the posterior error covariance associated with the proposed RSKF. Since the filter is developed for the RRIPS, it suffices to analyze the covariance matrix $\Pb_{k|k}$ of the transformed system. Furthermore, it is known that $\Pb_{k|k}$ remain positive definite for all $k \ge 0$ under standard filtering assumptions \cite{jazwinski2007stochastic}.

\begin{thm}
Assume that the system \eqref{mcsys} is uniformly completely observable and uniformly completely controllable, with $\Pb_0 > 0$. Furthermore, let the RS parameter be chosen such that $2\mu_{1,k-1}\Pb_{k-1|k-1}<I$. Then the posterior error covariance matrix $\Pb_{k|k}$ remains uniformly bounded for all $k \geq 1$, provided that the uncertainty satisfies $-I < O_{k,k-l}^{-1} \Delta O_{k,k-l} < I$, where $\Delta O_{k,k-l} = \sum \limits _{i=k-l}^{k}\A_{i,k}^\top \Hb_{m_k}^\top (R_i^{m_k})^{-1}\Hb_{m_k} \Delta \A_{i,k}$.
\end{thm}

\begin{proof} 
Considering the assumptions stated above and following the Appendix B of \cite{tiwari2022risk}, the above theorem can be proved.
\end{proof}

\section{Simulation Results}\label{sec:simu}
Consider the discrete-time stochastic singular system described by \eqref{sls}, where the associated system matrices are given by
\begin{equation*}
\begin{aligned}
E &= \begin{bmatrix} 1 & 0 & 0 \\ 
    0 & 1 & 0 \\
    0 & 0 & 0 \end{bmatrix}, ~
A = \begin{bmatrix} 0.99 & 0.01 & 0 \\
0 & 0.99 & 0 \\
1 & -1 & -1 \end{bmatrix},~B = \begin{bmatrix} 0.04 \\ 0.02 \\ 0.05 \end{bmatrix}, ~\text{and} \\
H &= \begin{bmatrix} 1 & 0 & 1 \\
0 & 1 & 0 \end{bmatrix}.
\end{aligned}
\end{equation*}
It can be readily verified that $E$ is singular and the matrix pair $(E,A)$ satisfies the regularity condition. The model uncertainty is taken as $\Delta A = \operatorname{diag}(\delta,\delta,\delta)$, where $\delta$ represents the uncertainty level. The process and measurement noise, denoted by $w_k$ and $v_k$, are assumed to be mutually independent, zero mean Gaussian sequences with covariance matrices $Q = 0.09$ and $R = \operatorname{diag}(0.5,0.5)$, respectively. The initial state vector is chosen as $x_0 = [ 1 \; 3 \; -2]^\top$. Applying the transformation procedure outlined in Section \ref{subsec:systrans} yields the transformed state components $x_{1,k}$ and $x_{2,k}$. Subsequently, by augmenting these components, the singular system can be represented in the augmented form \eqref{asys}. For the implementation of the proposed filter, the initial estimates of the transformed states are generated from a standard Gaussian distribution while the initial error covariance matrix is selected as $\Pb_{0|0} = \operatorname{diag}(9,1,4)$. 

To validate the performance of the proposed RSKF, comparisons are carried out with the standard KF using the root mean square error (RMSE) and the time averaged mean square error (Avg-MSE) as evaluation metrics. All simulation results are obtained from $500$ Monte Carlo runs. The RS parameter $\mu_{1,k-1}$, is selected adaptively so that it remains below the smallest positive root of $| \Pb^{-1}_{k-1|k-1} - 2\mu_{1,k-1} I| = 0$, thereby ensuring that $(\Pb_{k-1|k-1}^{-1}-2\mu_{1,k-1}I)$ remains positive definite and that the covariance recursion is well defined.

Figs.~\ref{fig:1} and \ref{fig:2} present the RMSE trajectories of the first and second state components for the uncertainty level $\delta = 0.01$. It can be observed that the proposed RSKF produces consistently smaller RMSE values than the standard KF and demonstrating its enhanced estimation capability in the presence of modeling uncertainty. To further examine the robustness of the proposed approach, the Avg-MSE is computed for $\delta \in [-0.01, 0.01]$, and the corresponding results are shown in Fig.~\ref{fig:3}. As the uncertainty magnitude increases, the RSKF maintains lower Avg-MSE values than the KF, whereas the performance of both filters becomes nearly indistinguishable when $\delta$ is close to zero. Consequently, these results highlight the improved robustness and estimation performance of the proposed RSKF under uncertain system dynamics.

\begin{figure}[H]
\centering
\includegraphics[width = 0.45\textwidth]{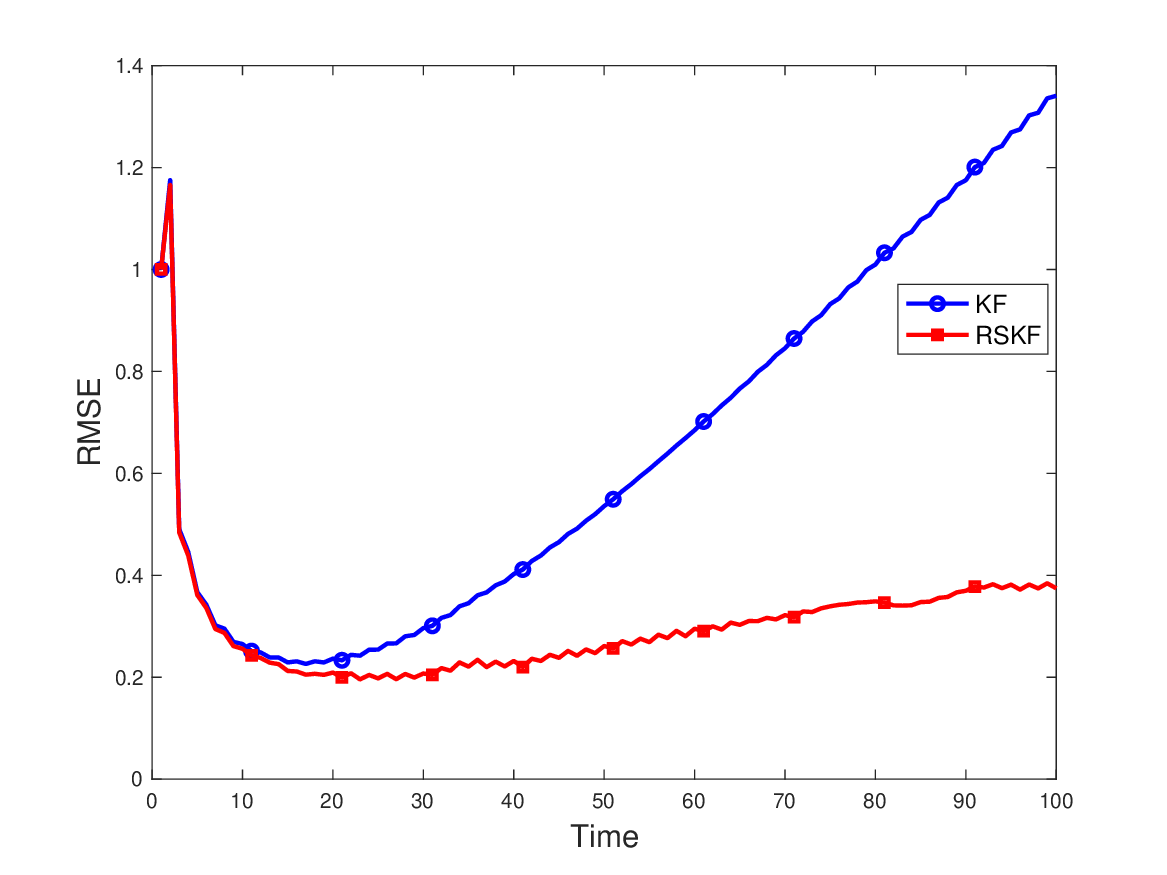}
\caption{RMSE comparison for first state of $x$}
\label{fig:1}
\end{figure}

\begin{figure}[H]
\centering
\includegraphics[width = 0.5\textwidth]{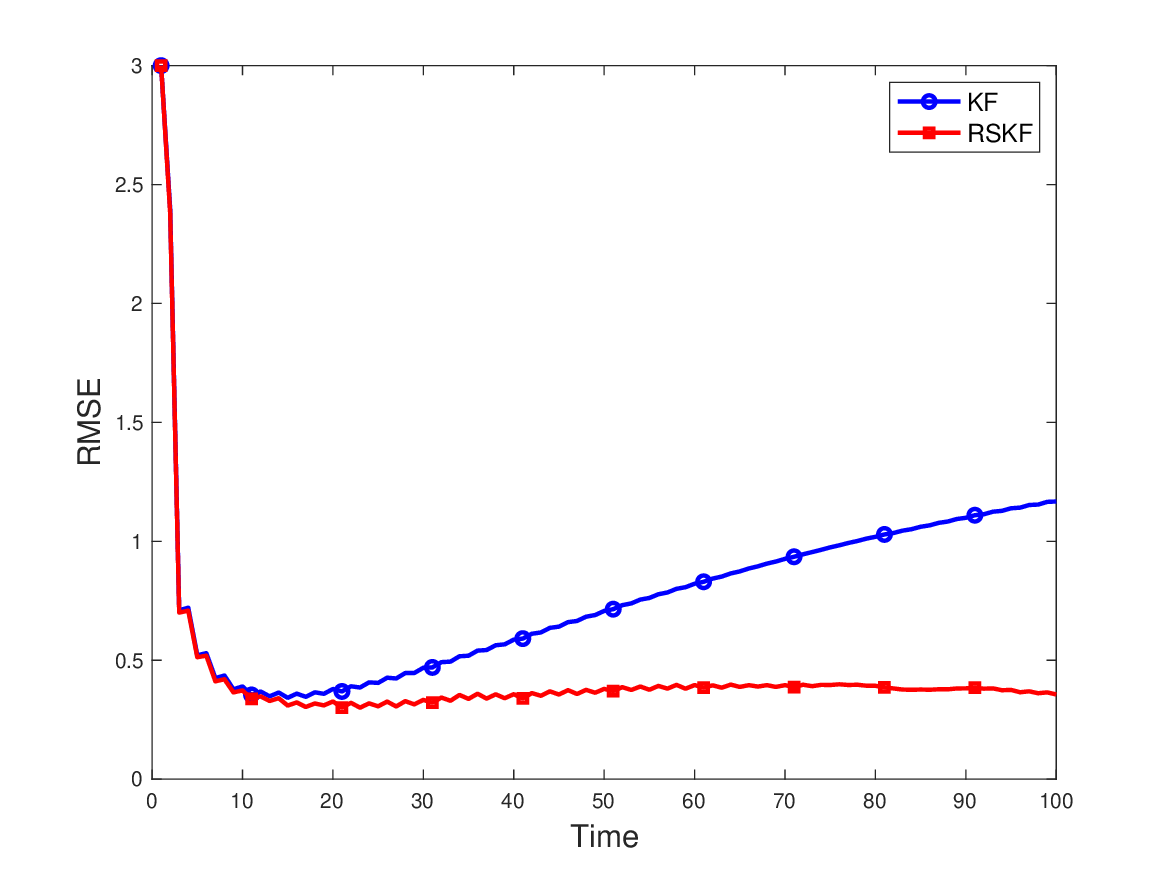}
\caption{RMSE comparison for second state of $x$}
\label{fig:2}
\end{figure}

\begin{figure}[H]
\centering
\includegraphics[width = 0.5\textwidth]{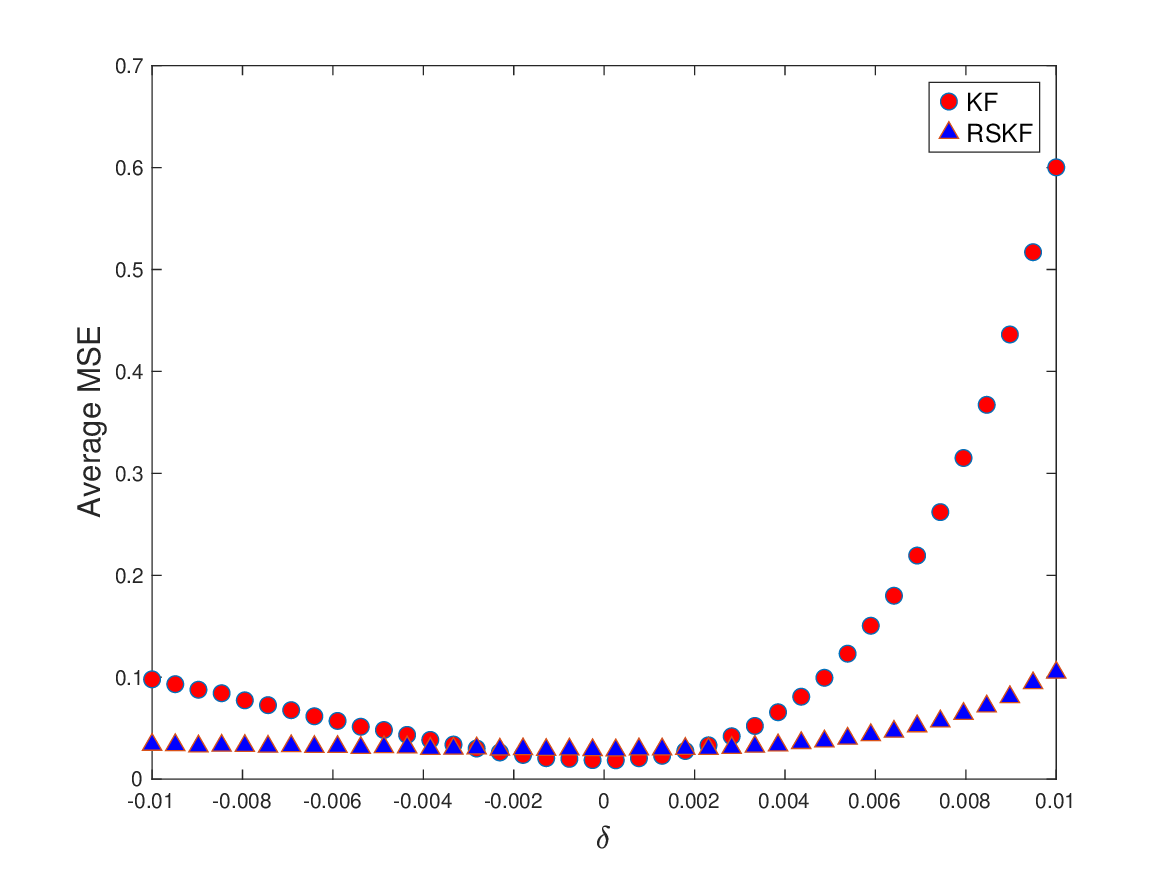}
\caption{Average MSE comparison for first state of $x$}
\label{fig:3}
\end{figure}

\section{Conclusion}\label{sec:concl}
This paper has developed a RS Kalman filtering framework for linear discrete-time networked singular systems operating under RRP. By employing the WCF, the singular system is transformed into an equivalent RRIPS, which enables the derivation of recursive RS filtering equations through a Bayesian formulation and an exponential quadratic cost criterion. An adaptive RS mechanism is further introduced to enhance robustness against modeling uncertainties while preserving the positive definiteness of the predicted covariance matrix. Moreover, observability and stability properties of the proposed filer are established using the uniformly complete observability and controllability criteria. Numerical results demonstrated that the proposed RSKF provides improved robustness and estimation performance compared with the standard KF under communication constraints and model uncertainty. Future research will focus on extending the proposed framework to networked singular systems under more general communication constraints such as one-step or finite-step packet dropouts, random measurement delays, and missing measurements.

\bibliographystyle{ieeetr}       
\bibliography{autosam}

\end{document}